\documentclass[12pt,reqno]{amsart}
\usepackage{amssymb}
\usepackage[T1]{fontenc}
\usepackage[latin1]{inputenc} 
\usepackage{dsfont, mathrsfs}
\usepackage{a4wide}  
\usepackage{stmaryrd}
\usepackage{amsthm}
\usepackage{graphicx}
\usepackage{color,soul} 
\usepackage{xspace}
\usepackage[hyphenbreaks]{breakurl}
\usepackage[hyphens]{url}

\usepackage[colorlinks, pagebackref, hyperindex]{hyperref}

\hypersetup{
    colorlinks=true,         
    linkcolor=blue,          
    citecolor=green,         
    filecolor=magenta,       
    urlcolor=cyan            
}

\newtheorem{theorem}{Theorem}[section]
\newtheorem{lemma}[theorem]{Lemma}

\newtheorem{corollary}[theorem]{Corollary}

\theoremstyle{definition}

\newtheorem{question}[theorem]{Question}

\theoremstyle{remark}
\newtheorem{remark}[theorem]{Remark}

\newcommand{\Cc}{\mathbb{C}}
\newcommand{\Dd}{\mathbb{D}}
\newcommand{\Ee}{\mathbb{E}}

\newcommand{\Nn}{\mathbb{N}}
\newcommand{\Pp}{\mathbb{P}}

\newcommand{\Rr}{\mathbb{R}}

\newcommand{\Uu}{\mathbb{U}}

\newcommand{\Un}{\mathds{1}} 

\newcommand{\Be}{\mathcal{B}}

\newcommand{\Fe}{\mathcal{F}}
\newcommand{\Ge}{\mathcal{G}}

\newcommand{\Le}{\mathcal{L}}

\newcommand{\Ig}{\mathfrak{I}}
\newcommand{\Rg}{\mathfrak{R}}
\newcommand{\Sg}{\mathfrak{S}}

\newcommand{\eg}{\mathfrak{e}}
\newcommand{\mg}{\mathfrak{m}}

\newcommand{\Lambdab}{{\boldsymbol{\Lambda}}}

\newcommand{\gammab}{{\boldsymbol{\gamma}}}

\newcommand{\thetab}{{\boldsymbol{\theta}}}

\newcommand{\ab}{{\boldsymbol{a}}}
\newcommand{\bb}{{\boldsymbol{b}}}

\newcommand{\ub}{{\boldsymbol{u}}}

\newcommand{\xb}{{\boldsymbol{x}}}

\newcommand{\zb}{{\boldsymbol{z}}}

\def\inv{^{-1}} 
\def\longlongrightarrow{\hspace{+0.1ex} - \hspace{-1.1ex} - \hspace{-1.1ex} - \hspace{-1.1ex}\longrightarrow  } 

\def\geq{\geqslant}
\def\leq{\leqslant}

\def\Re{\Rg \eg}
\def\Im{\Ig \mg}

\let\oldforall\forall
\def\forall{\oldforall\,}
\let\oldexists\exists
\def\exists{\oldexists\,}

\newcommand{\ensemble}[1]{ \left\lbrace #1 \right\rbrace }   
\newcommand{\prth}[1]{\!\left( #1 \right) }  
\newcommand{\Esp}[1]{ \Ee \prth{ #1 } }  
\newcommand{\Espr}[2]{ \Ee_{#1} \prth{ #2 } }  
\newcommand{\Prob}[1]{ \Pp \prth{ #1 } }

\newcommand{\crochet}[1]{\!\left[ #1 \right] }  
\newcommand{\intcrochet}[1]{\llbracket #1 \rrbracket} 
\newcommand{\abs}[1]{\left| #1 \right|}  
\newcommand{\norm}[1]{\left| \! \left| #1 \right| \! \right|} 

\newcommand{\Unens}[1]{ \Un_{ \ensemble{#1} } }

\newcommand{\tendvers}[2]{ \underset{#1 \rightarrow #2}{\longlongrightarrow} } 
 
\newcommand{\equivalent}[2]{ \underset{#1 \rightarrow #2}{\sim} }

\def\eqlaw{\stackrel{\Le}{=}}
\def\per{ \operatorname{per} }
\def\Bergman{ \Be }
\def\divise{ | }
\def\tr{ \operatorname{tr} }
\def\Ai{ \operatorname{Ai} }

\definecolor{rougeclair}{rgb}{1,.65,.65}
\sethlcolor{rougeclair}

\newcommand{\emailhref}[1]{ \email{\href{mailto:#1}{#1}} }

\title[On the largest root modulus of a Kac polynomial]{Extreme value statistics for the roots of a complex Kac polynomial}
\author[Y. Barhoumi-Andr\'eani]{Yacine Barhoumi-Andr\'eani}
\address{Department of Statistics, University of Warwick, Coventry CV4 7AL, U.K.}
\emailhref{y.barhoumi-andreani@warwick.ac.uk}
\date{\today}


\begin{document}

\begin{abstract}
We investigate the fluctuations and large deviations of the root of largest modulus in a model of random polynomial with independent complex Gaussian coefficients (Kac polynomials). The fluctuations were recently computed by R. Butez (arxiv 1704.02761) and involve a Fredholm determinant. The precise large deviations involve a function defined by a series of mutiple integrals similar to such an expansion and exhibit a link with moments of characteristic polynomials of truncated Haar-distributed random unitary matrices.
\end{abstract}

\maketitle

\section{Introduction}

Consider the following random polynomial
\begin{align}\label{Def:ModelPolAl}
P(z) := \sum_{k = 0}^n G_k z^k = G_n \prod_{k = 1}^n (z - Z_{k, n})
\end{align}
where $ (G_k)_{0 \leq k \leq n} $ is a sequence of i.i.d. Gaussian random variables. These polynomials are commonly called Kac polynomials since their seminal study by Kac (\cite{KacPAG1} ; see also \cite[ch. 1.1]{DemboPoonenShaoZeitouni} for historical remarks). The study of their roots has given rise to numerous interesting questions, concerning for instance the number of roots in a given set (in the case where the $ G_k $'s are real, a classical question concerns the number of real roots of such a polynomial, see \cite{EdelmanKostlan}) or more generally the linear statistics of the roots, i.e. random variables of the form $ \frac{1}{n} \sum_{k = 1}^n f(Z_{k, n} ) $ for a given function $f$ (see the introduction of \cite{ButezPAG} and cited references for a review of recent results).

In the case of complex Kac polynomials, a lot is known about the repartition of the roots, in particular, one knows that the roots cluster uniformly around the unit circle (see e.g. \cite{HughesNikeghbaliPol, IbragimovZaporozhets}). Nevertheless, this does not preclude some particular roots to escape arbitrarily far away from this area and a natural question is thus to investigate if the maximum modulus of the roots is concentrated around $1$. Such a behaviour was recently investigated by R. Butez (\cite{ButezPAG}) who addresses the question of fluctuations of the extreme value statistics of the roots and answers it in the following theorem~:

\begin{theorem}[Butez]\label{Theorem:Butez} Order the sequence of roots $ (Z_{k, n})_{1 \leq k \leq n} $ according to their modulus, so that $ \abs{Z_{1, n} } < \cdots < \abs{Z_{n, n} } $ a. s. Then, $ (Z_{1, n})_n $ converges in distribution to a random variable $ Z^* $ a. s. inside the unit disk and $ (Z_{n, n})_n $ converges in law to $ 1/Z^* $. Moroever, the point process $ \ensemble{ Z_{k, n}, \abs{Z_{k, n} } < 1 } $ converges in distribution in the space of Radon measures\footnote{We associate a point process $ (X_k)_k $ to a point measure $ \sum_k \delta_{X_k} $, see \cite{ButezPAG}.} to the roots of the Gaussian analytic function $ \prth{ \sum_{n \geq 0} G_n z^n }_{ \abs{z} < 1 } $ whose law is given by a determinantal point process with Bergman kernel (the zeroes of such an analytic function are countable with no accumulation points a. s.). 
\end{theorem}

The determinantal character of the roots of such a random analytic function is a direct consequence of a theorem of Peres and Vir\`ag (see \cite{PeresVirag}), a corollary of which is the explicit computation of the law of $ \max_{1 \leq k \leq n } \abs{Z_{n, k} } $ when $ n \to \infty $ :

\begin{theorem}[Peres-Vir\`ag] Let $ (U_k)_{k \geq 1} $ be a sequence of i.i.d. random variables uniformly distributed on the interval $ \crochet{0, 1} $. Then, one has the following equality in distribution
\begin{align*}
\frac{1}{Z^*} \eqlaw \max_{k \geq 1} U_k^{-1/2k}
\end{align*}
\end{theorem}

As one can see, the root of maximum modulus does not converge to $1$. One can check moreover that
\begin{align*}
\Prob{\max_{k \geq 1} U_k^{-1/2k} \leq y } = \prod_{ k \geq 1 } \prth{1 - y^{-2k } } \Unens{ y > 1 }
\end{align*}
namely, its law is supported on $ (1, +\infty) $. This fact raises the natural question of the large deviations of $ \max_{1 \leq k \leq n } \abs{Z_{n, k} } $ inside the unit circle. This is the main theorem of the paper.

\begin{theorem}[Precise left large deviations]\label{Theorem:MainTheorem} Let $ y \in (0, 1) $. Then
\begin{align*}
\Prob{ \max_{1 \leq k \leq n } \abs{Z_{n, k} } \leq y }  =  \frac{ y^{n(n + 1) } }{ n^{n + 1 } } \prth{ \Fe(y) + O_y\prth{ \frac{1}{ n } } }
\end{align*}
with
\begin{align*}
\Fe(y) := \sum_{k \geq 0} \frac{ (-1)^{ \frac{ k( k + 1)}{2} }    }{ k! \,(k + 1)! }    \int_{ \Uu^k } \prod_{ i, j = 1}^k \frac{   1  }{  1 -  y^2 u_i \overline{u_j} }  \prod_{\ell = 1}^k  \frac{d u_\ell}{ 2 i \pi  u_\ell }
\end{align*}
\end{theorem}

The plan of this article is the following : we start by recalling some classical facts about Kac polynomials, and we give an alternative proof to Butez' theorem \ref{Theorem:Butez}, we then treat in an elementary way the left large deviations (in the logarithmic setting, thus) and we finally prove theorem \ref{Theorem:MainTheorem}. The proof involves the characteristic polynomial of a model of truncated random unitary matrices distributed according to the Haar measure (the so-called Circular Unitary Ensemble). We conclude with questions of interest, perspectives and future work on the topic.

\section*{Notations and properties}

We gather here some notations used throughout the paper. The unit circle will be denoted by $ \Uu $, the unit disk by $ \Dd $ and the disk of radius $y$ by $ \Dd(y) $, i.e. 
\begin{align*}
\Dd(y) := \ensemble{ z \in \Cc \ : \ \abs{z} \leq y }
\end{align*}

When integrating on $ \Uu $ or any curve in the complex plane, we set 
\begin{align*}
d^*u := \frac{du}{2 i \pi}
\end{align*}

The Stirling formula is
\begin{align}\label{Eq:StirlingFormula}
\log(n!) = n \log(n) - n + \frac{1}{2} \log(2\pi n) + O\prth{ \frac{1}{n} } \ \ \ \Longleftrightarrow \ \ \ n! = e^{ O(1/n) } \prth{ \frac{n}{e} }^n \sqrt{2\pi n}
\end{align}

We define the Vandermonde determinant, for $ \ab := (a_1, \dots, a_n) $, by 
\begin{align*}
\Delta(\ab) := \prod_{1 \leq i < j \leq n } (a_i - a_j)
\end{align*}

In particular, for $ u_j \in \Uu $, one has
\begin{align}\label{Eq:VandermondeDuality}
\Delta(\ub) & = \prod_{1 \leq i < j \leq k} (u_j - u_i) = \prod_{1 \leq i < j \leq k} u_i u_j(u_i\inv - u_j\inv) = \overline{\Delta(\ub) } \prod_{1 \leq i < j \leq k} (- u_i u_j) \notag \\
             & = (-1)^{ k(k + 1)/2 } \prod_{1 \leq  j \leq k}  u_j^{k - 1}  \, \overline{\Delta(\ub) } 
\end{align}

For $ k \geq 0 $, we define the $k$-th elementary symmetric polynomial by $ e_0(\ab) := 1 $ and 
\begin{align*}
e_k(\ab)  := \sum_{1 \leq i_1 < i_2 < \cdots < i_k \leq n}  a_{i_1} a_{i_2} \cdots  a_{i_k}  
\end{align*}

We recall that for all $ t \in \Cc $, we have the following equality (see e.g. \cite{MacDo})
\begin{align}\label{Eq:GeneratingElementary}
\sum_{k = 0 }^n e_k(\zb) t^k = \prod_{k = 1}^n \prth{ 1 + t z_k }
\end{align}
%
%
%

The Cauchy determinant is given by 
\begin{align}\label{Def:CauchyDeterminant}
\det\prth{ \frac{1}{a_i + b_j } }_{1 \leq i, j \leq k} =   \frac{ \Delta(\ab)  \Delta(\bb)  }{ \prod_{1 \leq i, j \leq k} (a_i + b_j ) } 
\end{align}

Last, for a matrix $ A $, we will write $ A\crochet{i, j} $ for its $ (i, j) $-coefficient.

\section{Fluctuations}

\subsection{Reminders on random Gaussian polynomials} From now on, we denote by $ (Z_k)_k $ in place of $ (Z_{k, n})_k $ the roots of the complex Kac polynomial \eqref{Def:ModelPolAl}.

The \textit{correlation functions} or \textit{joint intensities} of the roots $ (Z_k)_{1 \leq k \leq n} $ in \eqref{Def:ModelPolAl} are defined by  

\begin{align*}
\rho^{(n)}_k(z_1, \dots, z_k) := \Esp{ \sum_{ i_1 \neq \cdots \neq i_k \in \intcrochet{1, n} } \prod_{ \ell = 1 }^k \delta_0\prth{ Z_{i_\ell} - z_\ell } }  
\end{align*}

Let $ f_Z(x) $ denote the Lebesgue-density of the random variable $ Z $. Using the Hammersley formula (see e.g. \cite[(3.3.1) p. 39]{BenHougKrishnapurPeresVirag}) that writes, for any random polynomial $ P $
\begin{align*}
\rho_{P, k} (z_1, \dots, z_k) & = \Esp{ \prod_{\ell = 1}^k \abs{ P'(z_\ell) }^2 \delta_0\prth{ P(z_\ell) } } \\
                 & =  f_{(P(z_1), \dots, P(z_k) ) }(0, \dots, 0)\ \Esp{ \prod_{\ell = 1}^k \abs{ P'(z_\ell) }^2 \bigg\vert P(z_1) = \dots = P(z_k) = 0 } 
\end{align*}
one gets the formula (see \cite[(3.4.2) p. 40]{BenHougKrishnapurPeresVirag} and references cited)
\begin{align}\label{Eq:RootsCorrelations}
\rho_k^{(n)}(z_1, \dots, z_k) =  \frac{ \per\prth{ C_{k,n}(\zb) - B_{k,n}(\zb) A_{k,n}(\zb)\inv B_{k,n}(\zb)^* } }{ \det( \pi A_{k,n}(\zb) ) } 
\end{align}
where $ A_{k,n}(\zb), B_{k,n}(\zb) $ and $ C_{k,n}(\zb) $ are the $ k \times k $ matrices of general term given by 
\begin{align}\label{Def:CovariancesPolAndDerivatives}
\begin{aligned} 
A_{k,n}(\zb)\crochet{i, j} & := \Esp{ P(z_i) \overline{P(z_j) } } = \frac{1 - x^n}{1 - x} \bigg\vert_{x = z_i \overline{z_j}}  =: h_n(z_i \overline{z_j})  \\
B_{k,n}(\zb)\crochet{i, j} & := \Esp{ P(z_i) \overline{P'(z_j) } } = \frac{d}{dx} \frac{1 - x^n}{1 - x} \bigg\vert_{x = z_i \overline{z_j} } \times z_i =: z_i g_n(z_i \overline{z_j}) \\
C_{k,n}(\zb)\crochet{i, j} & := \Esp{ P'(z_i) \overline{P'(z_j) } } = \prth{\frac{d}{dx} }^2 \frac{1 - x^n}{1 - x} \bigg\vert_{x = z_i \overline{z_j}} \times z_i \overline{z_j} =: z_i \overline{z_j} f_n(z_i \overline{z_j})
\end{aligned}
\end{align}

Let $ (X_k)_{k \geq 1} $ be a sequence of random variables with values in $ \Rr $ and with correlation functions $ (\rho_{X, k})_k $. Then, if for a set $ A \subset \Rr $ one has 
\begin{align*}
\sum_{k \geq 0} \frac{1}{k! } \int_{ A^k } \abs{ \rho_{X, k}(x_1 \dots, x_k) } dx_1 \dots dx_k < \infty
\end{align*}
then, using an inclusion-exclusion formula (see e.g. \cite{BorodinPPDSurvey}), one gets 
\begin{align}\label{Eq:MaxWithCorrels}
\Prob{ \forall k \geq 1, \ X_k \notin A } = \sum_{k \geq 0} \frac{ (-1)^k  }{k! } \int_{ A^k }  \rho_{X, k}(x_1 \dots, x_k)  dx_1 \dots dx_k  
\end{align}

We will use this formula to compute the limiting distribution of the maximum modulus of \eqref{Def:ModelPolAl}.

\subsection{The limiting distribution of the maximum modulus}

We define the maximum modulus of the roots by
\begin{align}\label{Def:MaximumModulus}
\rho_{n} := \max_{1 \leq k \leq n } \abs{ Z_k }
\end{align}

The following gap probability for $ y > 1 $ 
\begin{align*}
\Prob{\rho_{n} \leq y } = \Prob{ \forall k \in \intcrochet{1, n}, \ \abs{Z_k} \leq y } 
\end{align*}
is known to converge to (see \cite{ButezPAG})
\begin{align*}
\Prob{ \forall k \geq 1 , \ \abs{Y_k} \geq y\inv } = \det\prth{ I - \Bergman }_{ L^2( \Dd(y\inv) ) }
\end{align*}
where $ \Bergman $ is the operator of kernel $ B(z, z') :=   (1 - z \overline{z'})^{-2} $ (Bergman kernel) acting on $ L^2( \Dd, dz/\pi  ) $ and where the process $ (Y_k)_{k \geq 1} $ is determinantal of kernel $ \Bergman $.

This is a consequence of the convergence of the polynomial $ (P_n(z))_{z \in \Dd(r) } $ to the Gaussian Analytic Function $ (P_\infty(z))_{z \in \Dd(r) } := ( \sum_{k \geq 0} G_k z^k )_{z \in \Dd(r) } $ for all $ r < 1 $ and the measurability of the map $ P \mapsto \max_{P(x) = 0} \abs{x} $ (see \cite{ButezPAG}). We moreover have for all $ y \in \Rr_+ $
\begin{align}\label{Eq:LimitLaw}
\Prob{\rho_{n} \leq y } \tendvers{n }{ +\infty } \Prob{ \frac{1}{\min_{k \geq 1} \ensemble{ U_k^{1/2k} } } \leq y } = \prod_{ k \geq 1 } \prth{ 1 - y^{-2k } } \Unens{ y > 1 }
\end{align}
where $ (U_k)_{k \geq 1} $ is a sequence of i.i.d. uniform random variables on $ \crochet{0, 1} $.

We now prove this result in a more analytic way. 

\begin{lemma} The correlation functions $ \rho_k^{(n)} $ of the roots of the Gaussian polynomial \eqref{Def:ModelPolAl} converge for all $k$, when $ n \to \infty $, uniformly on $ \Dd(y)^k $ for all $ y < 1 $. 
\end{lemma}


\begin{proof}
Using the formula \eqref{Eq:RootsCorrelations}, one sees that it is enough to prove the convergence of the covariances given in \eqref{Def:CovariancesPolAndDerivatives} in the underlying domain, namely, to prove that $ f_n $, $h_n$ and $g_n$ defined in \eqref{Def:CovariancesPolAndDerivatives} converge inside the unit disc. One has, uniformly in $ z \in \Dd(y) $ with $ y < 1 $
\begin{align*}
h_n(z) := \frac{1 - z^n}{1 - z} \tendvers{n}{+\infty } \frac{1}{1 - z} =: h_\infty(z)
\end{align*}

Moreover, 
\begin{align*}
g_n(z) = h_n'(z) = \frac{ 1 + (n - 1)z^n - n z^{n - 1} }{ (1 - z)^2 }  \tendvers{n}{+\infty } \frac{1}{ (1 - z)^2 } =: g_\infty(z)
\end{align*}
and 
\begin{align*}
f_n(z) & = h_n''(z) = \frac{  2 - (n-1)(n - 2) z^n + 2n(n - 2) z^{n - 1}  - n(n - 1) z^{n - 2} }{ (1 - z)^3 }  \\
       & \tendvers{n}{+\infty } \frac{ 2 }{ (1 - z)^3 } =: f_\infty(z)
\end{align*}

As a result, one has the convergence, uniformly in $ \Dd(y)^k $
\begin{align*}
\rho_k^{(n)}(z_1, \dots, z_k) \tendvers{n}{+\infty } \rho_k^{(\infty )}(z_1, \dots, z_k)
\end{align*}
with 
\begin{align}\label{Def:CorrelsLimites}
\rho_k^{(\infty )}  := \frac{ \per( C_{k,\infty} - B_{k,\infty} A_{k,\infty}\inv B_{k,\infty}^* ) }{ \det(\pi A_{k,\infty}  ) }
\end{align}
\end{proof}


\begin{corollary} The following convergence is satisfied for all $ y \geq 0 $
\begin{align}\label{Thm:LimitLaw}
\Prob{\rho_{n} \leq y } \tendvers{n }{ +\infty }  \prod_{ k \geq 1 } \prth{ 1 - y^{-2k } } \Unens{ y > 1 }
\end{align}

\end{corollary}

\begin{proof}
The convergence of the extreme values of the roots of \eqref{Def:ModelPolAl} is then a consequence of the following result from \cite{PeresVirag} : the limiting correlation functions \eqref{Def:CorrelsLimites} take the form 
\begin{align*}
\rho_k^{(\infty)}(z_1, \dots, z_k) = \det\prth{ g_\infty (z_i \overline{z_j}) }_{1 \leq i, j \leq k}
\end{align*}

This theorem uses the invariance of the limiting analytic function $ (\sum_{k \geq 0} G_k z^k)_{z \in \Dd} $ by the group of homographies of $ \Dd $ and the Borchardt's identity (see \cite[5.1.12]{BenHougKrishnapurPeresVirag}).

Writing $ \ensemble{ \max_k \abs{Z_k} < y } = \{ \min_k \abs{Z_k}\inv > y\inv \} $ and $ (Z_k)_k \eqlaw (Z_k\inv)_k $ due to the self-reciprocity of the polynomial, i.e. $ (z^n P(z\inv))_{ \abs{z} < 1 } \eqlaw (P(z))_{ \abs{z} > 1 }  $, one has, for all $ y > 1 $
\begin{align*}
\Prob{\rho_{n} \leq y } & = \sum_{k = 0}^n  \frac{1}{k! } \int_{ \Dd(y\inv)^k }  \rho_k^{(n)}(z_1 \dots, z_k)  dz_1 \dots dz_k \\
                & \tendvers{n}{+\infty } \sum_{k \geq 0} \frac{1}{k! } \int_{ \Dd(y\inv)^k }  \rho_k^{(\infty)}(z_1 \dots, z_k)  dz_1 \dots dz_k = \det\prth{ I - \Be }_{ L^2( \Dd(y\inv) ) }
\end{align*}

As a corollary, the set of modulus of the roots converge in distribution of an i.i.d. sequence $ (U_k^{1/2k})_{k \geq 1} $ where $ (U_k)_k $ is a sequence of i.i.d. random variables uniformly ditributed on $ \crochet{0, 1} $ (see \cite[th. 4.7.1 \& cor. 5.1.7]{BenHougKrishnapurPeresVirag}).

In the case of the maximum modulus of the roots, one can give a more analytical proof of this last result by finding the eigenvectors of the Bergman kernel acting on $ L^2(\Dd(t), dz) $ with $ t < 1 $. Indeed, setting $ \ensemble{\lambda_k(t) }_{k \geq 1} $ for the set of eigenvalues of the compact operator $ \Be $ acting on $ L^2( \Dd(t) ) $. The Bergman kernel is a trace-class operator on $ L^2(\Dd(t)) $ since $ \tr_{L^2( \Dd(t) ) }( \abs{ \Be }) = \int_{ \Dd(t) } \abs{ 1 - z \overline{z} }^{-2} \frac{dw}{\pi t^2 } < \infty $. Thus, one has (see e.g. \cite{GohbergGoldbergKrupnik})
\begin{align*}
\det\prth{ I - \Be }_{ L^2( \Dd(t) ) } = \prod_{k \geq 1} (1 - \lambda_k(t) )
\end{align*}
%
%
%

Consider the functions $ f_{k, t} \equiv f_k : z \in \Dd(t) \mapsto z^k $ for $ t < 1 $. Then, 
\begin{align*}
\Be f_{k, t}(z)   & = \int_{ \Dd(t) } (1 - z \overline{w})^{-2} w^k \frac{dw}{\pi t^2 } \\
                  & = \sum_{\ell \geq 0} \ell z^{\ell - 1 } \int_{ \Dd(t) } \overline{w}^{\ell - 1} w^k \frac{dw}{\pi t^2 } \\
                  & = \sum_{\ell \geq 0} \ell z^{\ell - 1 } 2\pi \Unens{k = \ell - 1} \int_0^t r^{\ell - 1} r^k \frac{r dr}{\pi t^2 } \\
                  & = (k + 1) z^k 2\pi \int_0^t r^{2k}   \frac{r dr}{\pi t^2 } = (k + 1) z^k 2 \pi \frac{t^{ 2k + 2 } }{2k + 2}   \frac{1}{\pi t^2 } \\
                  & = t^{2k} z^k = t^{2k} f_{k, t}(z) 
\end{align*}
hence the result.
\end{proof}

\begin{remark}
As the correlation functions determine the process, this proof applies to other models of random Gaussian polynomials if one can prove the convergence of the covariances. In general, using the inclusion-exclusion \eqref{Eq:MaxWithCorrels} for $ A = (-\infty, y) $, one gets
\begin{align*}
\Prob{ \max_{k \geq 1} X_k \leq y } = \sum_{k \geq 0} \frac{ (-1)^k }{k! } \int_{ [y, +\infty[^k }  \rho_{X, k}(x_1 \dots, x_k)  dx_1 \dots dx_k 
\end{align*}
\end{remark}

\subsection{A remark on the Tracy-Widom distribution}

The Tracy-Widom distribution (see \cite{TracyWidom}) writes as a Fredholm determinant in the same vein as the previous probability and satisfies moreover
\begin{align*}
F_{TW_2}(y) =  \exp\prth{ - \int_y^{+\infty } H_{TW_2}(t) dt  }
\end{align*}
where $ H_{TW_2}(t) $ writes as
\begin{align*}
H_{TW_2}(t) := \int_t^{+\infty } q(s)^2 ds  =  \int_\Rr q(s)^2 \Unens{ s - t \geq 0  } ds = q^2 * \Un_{\Rr_+}(t)
\end{align*}
with $q$ the Hastings-McLeod solution of the Painlev\'e II equation $ q''(s) = sq(s) + q(s)^2 $ satisfying $ q(x) \sim \Ai(x) $ when $ x\to +\infty $, $ \Ai $ being the Airy function (see \cite{HastingsMcLeod}) and $*$ designates the additive convolution $ f*g(x) := \int_\Rr f(t) g(x - t) dt $.

We are interested in the equivalent form for the probability density \eqref{Thm:LimitLaw}. For this, we write for $ y > 1 $
\begin{align*}
-\log \Prob{ \rho_\infty \leq  y }  & = - \sum_{k \geq 1} \log\prth{ 1 - y^{-2k } } = \sum_{k \geq 1} \int_y^{+\infty } \frac{ 2 k s^{-2 k - 1}  }{1 - s^{-2k } } ds \\
                & = 2 \int_y^{+\infty } \sum_{k \geq 1} \sum_{\ell \geq 0}  k s^{-2 k - 1 - 2 k \ell} ds  \\
                & = 2 \int_y^{+\infty } \sum_{k \geq 1} \sum_{m \geq 1}  k s^{-2 k m} \frac{ds}{s }    = 2 \int_y^{+\infty } \sum_{d \geq 1} \prth{ \sum_{k, m \geq 1}  k \Unens{km = d} } s^{-2 d} \frac{ds}{s }    
\end{align*}

Define the sum of divisors of $ d \in \Nn^* $ by
\begin{align*}
\sigma(d) :=  \sum_{k, m \geq 1}  k \Unens{km = d} = \sum_{k \divise d } k
\end{align*}

Then, setting $ \Sg(s) := 2 \sum_{d \geq 1} \sigma(d) s^{-2 d}  $, one gets
\begin{align*}
-\log \Prob{ \rho_\infty \leq  y }  = 2 \int_y^{+\infty } \sum_{d \geq 1} \sigma(d) s^{-2 d} \frac{ds}{s }  =  \int_y^{+\infty }  \Sg(s) \frac{ds}{s }   
\end{align*}

Note that one can write
\begin{align*}
\Prob{ \rho_\infty \leq  y } = \exp\prth{ \int_1^{+\infty } \Sg(y t) \frac{dt}{t} } = \exp\prth{ \int_0^1 \Sg(y t\inv) \frac{dt}{t} } = \exp\prth{ \Sg \star \Un_{ \crochet{0, 1} } (y) }
\end{align*}
where $ \star $ designates the multiplicative convolution defined by $ f \star g (x) := \int_{\Rr_+} f(t) g(x t\inv) \frac{dt}{t} $.

We see that $ \Sg $ is the analogue of $ - q^2 * \Un_{\Rr_+} $. This motivates the following question :

\begin{question}
Can one find the same decomposition with the Tracy-Widom distribution, namely, can one find the eigenvalues of the Airy operator which is the analogue of the Bergman operator in the previous setting ?
\end{question}

\section{Large deviations and precise large deviations}

\subsection{The density of the zeroes}

The distribution of the roots vector $ (Z_k)_{1 \leq k \leq n} $ of the polynomial \eqref{Def:ModelPolAl} is given by (see e.g. \cite{BenHougKrishnapurPeresVirag, BogomolnyBohigasLeboeuf, ForresterHonner, Hammersley, Kostlan})
\begin{align}\label{RootsLaw}
\Prob{ Z_1 \in dz_1, \dots, Z_n \in dz_n } = \frac{n! }{\pi^n} \frac{ \abs{ \Delta(\zb) }^2 }{ \crochet{ \sum_{k = 0}^n \abs{ e_k(\zb) }^2 }^{n + 1} } d\zb 
\end{align}
where we have set $ \zb := (z_1, \dots, z_n) $ and $ d\zb := \prod_{k = 1}^n d\Re(z_k) \, d\Im(z_k) $.

The formula \eqref{RootsLaw} is obtained using the transformation $ (X_0, X_1, \dots, X_n) \mapsto (X_n, Z_1, \dots, Z_n) $ given by the well-known formula valid for all $ k \in \intcrochet{1, n} $
\begin{align*}
X_{n - k} = (-1)^k X_n \,  e_k(Z_1, \dots, Z_n)
\end{align*}

The Jacobian of this transformation is given by the Vandermonde determinant (see e.g. \cite{BenHougKrishnapurPeresVirag}). One then integrates on $ X_n $ to get \eqref{RootsLaw}, namely
\begin{align*}
\Prob{Z_1 \in dz_1, \dots, Z_n \in dz_n } & = \abs{ \Delta(\zb) }^2 \int_\Cc f_{(X_0, \dots, X_n) }\prth{ \vphantom{a^{a^a}} (-1)^n y e_n(\zb), \dots, y e_0(\zb) } \abs{y}^{2n} \, dy \, d\zb \\
                & = \abs{ \Delta(\zb) }^2 \int_\Cc \exp\prth{ - \sum_{k = 0}^n \abs{ (-1)^k y \,  e_k(\zb) }^2 } \abs{y}^{2n}  \frac{dy }{\pi }\, \frac{ d\zb}{\pi^n} \\
                & =  \abs{ \Delta(\zb) }^2 \int_\Cc \exp\prth{ - \abs{ y }^2 \sum_{k = 0}^n \abs{ e_k(\zb) }^2 } \abs{y}^{2n}  \frac{dy }{\pi }\, \frac{ d\zb}{\pi^n} \\
                & =  \frac{\abs{ \Delta(\zb) }^2}{ ( \sum_{k = 0}^n \abs{ e_k(\zb) }^2 )^{n + 1} } \int_\Cc \exp\prth{ - \abs{ t }^2  } \abs{t}^{2n}  \frac{ dt }{\pi }\, \frac{ d\zb}{\pi^n} \\
                & =  \frac{\abs{ \Delta(\zb) }^2}{ ( \sum_{k = 0}^n \abs{ e_k(\zb) }^2 )^{n + 1} } \, \frac{n! }{\pi^n} \, d\zb
\end{align*}

Denote by $ f_n $ the Lebesgue-density of $ (Z_1, \dots, Z_n) $ :
\begin{align}\label{DensiteRacinesIndep}
f_n(\zb) = \frac{n! }{\pi^n} \frac{ \abs{ \Delta(\zb) }^2 }{ \crochet{  \sum_{k = 0}^n \abs{ e_k(\zb) }^2 }^{n + 1} } 
\end{align}

Using the Plancherel-Parseval formula for a polynomial, i.e.  $ \sum_k \abs{a_k}^2 = \int_0^1 \abs{ \sum_k a_k e^{2 i \pi k \theta} }^2 d\theta $, and \eqref{Eq:GeneratingElementary} we get
\begin{align*}
E(\zb) := \sum_{k = 0}^n \abs{ e_k(\zb) }^2 = \int_0^1 \abs{ 1 + \sum_{k = 1}^{n}  (-1)^{k} e_k(\zb) e^{2i \pi k \theta }   }^2 d\theta = \int_0^1 \prod_{k = 1}^n   \abs{ e^{2i \pi  \theta } - z_k  }^2 d\theta 
\end{align*}

This transforms \eqref{DensiteRacinesIndep} into
\begin{align}\label{DensiteRacinesIndepInt}
f_n(\zb) = \frac{n! }{\pi^n} \frac{ \abs{ \Delta(\zb) }^2 }{ \crochet{  \int_0^1  \prod_{k = 1}^n   \abs{ e^{2i \pi  \theta } - z_k  }^2 d\theta  }^{n + 1} } 
\end{align}

\begin{remark}
This form was used in \cite{ZeitouniZelditch} to prove a large deviations result for the empirical measure of the roots.  
\end{remark}

We are now interested in the tail of this last probability. Using \eqref{RootsLaw} and \eqref{DensiteRacinesIndepInt}, we have 
\begin{align*}
\Prob{\rho_{n} \leq y } = \Prob{ \forall k \in \intcrochet{1, n}, \ \abs{Z_k} \leq y } = \int_{\Dd(y)^n}  f_n( \zb) d\zb  = \int_{\Dd^n}  f_n(y \zb) y^{2n} d\zb 
\end{align*}

Using the scaling property of the Vandermonde determinant
\begin{align*}
\Delta(\lambda z_1, \dots, \lambda z_n) = \lambda^{ \frac{n(n-1)}{2} } \Delta(z_1, \dots, z_n)
\end{align*}
one gets
\begin{align*}
\Prob{\rho_{n} \leq y }  & = \int_{\Dd^n}  f_n(y \zb) y^{2n} d\zb  \\
                & = y^{2n} \frac{ n!}{ \pi^n } \int_{\Dd^n} \frac{ \abs{ \Delta(y \zb) }^2 }{ \crochet{  \int_0^1  \prod_{k = 1}^n   \abs{ e^{2i \pi  \theta } - y z_k  }^2 d\theta  }^{n + 1} }   d\zb  \\
                & =  y^{n(n + 1)} n!  \int_{\Dd^n} \frac{  \abs{ \Delta(  \zb) }^2 }{ y^{ 2n(n+1) } \crochet{  \int_0^1  \prod_{k = 1}^n   \abs{ e^{2i \pi  \theta } y\inv -  z_k  }^2 d\theta  }^{n + 1} }   \frac{ d\zb }{ \pi^n } \\
                & =  n!  \int_{\Dd^n} \frac{  \abs{ \Delta(  \zb) }^2 }{  \crochet{  y^{ n } \int_0^1  \prod_{k = 1}^n   \abs{ e^{2i \pi  \theta } y\inv -  z_k  }^2 d\theta  }^{n + 1} }   \frac{ d\zb }{ \pi^n }
\end{align*}

A classical computation shows that 
\begin{align*}
\int_{ \Dd^n }   \abs{ \Delta( \zb ) }^2   d\zb = \pi^n
\end{align*}

One can thus define the probability measure
\begin{align}\label{Def:TruncatedCUE}
\Pp_n\prth{d\zb} := \Unens{ \zb \in \Dd^n } \abs{\Delta(\zb)}^2 \frac{d\zb}{\pi^n}
\end{align}

There is a matrix model underlying such a probability measure given by a $ n $-block truncation of a $  CUE(n + 1) $ random matrix, i.e. a random Haar-distributed unitary matrix of size $ (n + 1) \times (n + 1) $ (see \cite[eq. (16)]{SommersZycTruncate}). 

Using the formula valid for all $ x > 0 $ and $ n \in \Nn^* $
\begin{align*}
\frac{n!}{x^n } = \int_0^{+\infty } e^{-t x} t^n dt
\end{align*}
we can write this last probability as
\begin{align*}
\Prob{\rho_{n} \leq y }  =   \int_{ \Rr_+ } \Espr{n}{ e^{ -t y^n\!\! \int_0^1 \abs{ Z_n\prth{   e^{ 2i\pi \theta } y\inv } }^2 d\theta } } t^n dt
\end{align*}
where $ Z_n(x) $ is the characteristic polynomial of the random matrix $ M_n $ from the truncated $ CUE $ given by
\begin{align*}
Z_n(x) := \det\prth{ x I -  M_n }
\end{align*}

\subsection{Large deviations}

In the limiting fluctuations of the largest modulus, the support of the limiting law given in \eqref{Thm:LimitLaw} is defined on $ [1, +\infty) $. We are interested in the the large deviations of $ \rho_n $ inside the disk.

\begin{theorem}[Left large deviations] Let $ y \in (0, 1) $. Then, $ (\rho_n)_n $ satisfies a (left) large deviation principle at speed $ n^2 $ with rate function $ y \mapsto \log(y\inv) $, namely, for all $ y \in (0, 1) $
\begin{align*}
\frac{1}{n^2} \log \Prob{ \rho_n \! \leq y } \tendvers{n}{+\infty } - \log(y\inv)
\end{align*}

More precisely, one has, with a $ O $ independent of $y$
\begin{align*}
\frac{1}{n^2} \log \Prob{ \rho_n \! \leq y } = - \log(y\inv) + O\prth{ \frac{\log n}{n} }
\end{align*}
\end{theorem}


\begin{proof}

We have
\begin{align*}
\Prob{\rho_{n} \leq y }  = y^{ n(n + 1) } \Espr{n}{ \frac{n!}{ \crochet{  y^{2n}\!\! \int_0^1 \abs{ Z_n\prth{   e^{ 2i\pi \theta } y\inv } }^2 d\theta }^{n + 1} } }
\end{align*}

Define
\begin{align*}
\xi_n(y) & :=  \int_0^1 \abs{ Z_n\prth{ e^{-2i\pi \theta } y\inv } }^2 d\theta \\
\eta_n(y) & := y^{2n} \xi_n(y) = \int_0^1 \abs{ \det\prth{ I_n - y e^{ 2i\pi \theta } M_n  } }^2 d\theta
\end{align*}

We have
\begin{align*}
\eta_n(y) = \int_0^1 \abs{ \det\prth{ I_n - y e^{ 2i\pi \theta } M_n  } }^2 d\theta = \sum_{k = 0}^n y^{ 2k } \abs{ e_k(\Lambdab) }^2 \geq 1
\end{align*}
where $ \Lambdab $ are the eigenvalues of $ M_n $. This last inequality comes from the fact that $ e_0(\Lambdab) = 1 $. In particular, we have
\begin{align*}
\Espr{n}{ \frac{n!}{ \crochet{  y^{2n}\!\! \int_0^1 \abs{ Z_n\prth{   e^{ 2i\pi \theta } y\inv } }^2 d\theta }^{n + 1} } } \leq n!
\end{align*}
namely
\begin{align*}
\Prob{\rho_{n} \leq y }  = y^{ n(n + 1) } O\prth{ n! }
\end{align*}

Taking the logarithm and dividing by $ n^2 $, we get 
\begin{align*}
\frac{1}{n^2}\log \Prob{\rho_{n} \leq y }  & = \frac{n(n + 1)}{n^2} \log(y) +  \frac{ \log(n!)}{n^2} + O\prth{  \frac{1}{n^2} } \\
              & = - \log(y\inv) + \frac{ \log(n/e) }{n } +  \frac{ \log(y) }{n } + O\prth{ \frac{ \log(n) }{n^2} } \\
              & = - \log(y\inv) +  O\prth{  \frac{ \log(n)}{n } }
\end{align*}
where we have used the Stirling formula \eqref{Eq:StirlingFormula}.
\end{proof}

\subsection{Precise large deviations}

We now prove theorem \ref{Theorem:MainTheorem}.


\begin{proof}
Let $ A > 0 $ be a constant to be choosen later. Write
\begin{align*}
\Prob{\rho_{n} \leq y }  =   \prth{ \int_0^A + \int_A^{+\infty} \Espr{n}{ e^{ -t y^n\!\! \int_0^1 \abs{ Z_n\prth{   e^{ 2i\pi \theta } y\inv } }^2 d\theta } } t^n dt }
\end{align*}

Define 
\begin{align*}
F_n(y, A) :=  \int_0^A \Espr{n}{ e^{ -t y^n\!\! \int_0^1 \abs{ Z_n\prth{  e^{ 2i\pi \theta } y\inv  } }^2 d\theta } } t^n dt 
\end{align*}

Expanding the Laplace transform of the random variable 
\begin{align*}
\xi_n(y) := \int_0^1 \abs{ Z_n\prth{ e^{-2i\pi \theta } y\inv } }^2 d\theta
\end{align*}
and applying the Fubini theorem, one gets
\begin{align*}
F_n(y , A) & =  \sum_{k \geq 0} \frac{(-1)^k}{k! } \int_0^A t^{n + k} dt  \, y^{kn} \, \Espr{n}{ \prth{ \int_0^1 \abs{ Z_n\prth{ e^{ 2i\pi \theta  } y\inv } }^2 d\theta }^k } \\
              & =  \sum_{k \geq 0} \frac{ (-1)^k }{ k! } \frac{ A^{n + k + 1} }{n + k + 1 } \, y^{kn} \, \Espr{n}{ \prth{ \int_0^1 \abs{ Z_n\prth{  e^{ 2i\pi \theta   } y\inv } }^2 d\theta }^k }
\end{align*}

This last moment expansion is valid if the law of $ \xi_n(y) $ is defined by its moments. We now prove this fact. Using the Fubini theorem, we have
\begin{align*}
\Esp{ \xi_n(y)^k } & = \Espr{n}{ \prth{ \int_0^1 \abs{ Z_n\prth{ e^{ 2i\pi \theta  } y\inv } }^2 d\theta }^k } \\
                   & = \int_{ \crochet{0, 1}^k } \Espr{n}{ \prod_{\ell = 1}^k \abs{ Z_n\prth{ e^{ 2i\pi \theta_\ell   } y\inv } }^2 } d\thetab
\end{align*}

An explicit formula is available to compute the integer moments of the characteristic polynomial (see e.g. \cite[eq. (2.11)]{AkemannVernizzi} ; for the reader's convenience, this result is reminded in section \ref{Section:Annexe1}), namely
\begin{align}\label{Eq:MomentsCharPol}
\Espr{n}{ \prod_{\ell = 1}^k \abs{ Z_n(u_\ell) }^2   } = \frac{ \det\prth{ g_{n + k}(u_i \overline{v_j}) }_{1 \leq i, j \leq k } }{ \abs{ \Delta(\ub) }^2 } \times \frac{ n! }{ ( n + k )!}
\end{align}
where $ g_n $ is the function defined by
\begin{align}\label{Def:FunctionGn}
g_n(x) = \sum_{k = 0}^n k x^{k - 1} = \frac{d}{dx} \frac{1 - x^n}{1 - x} 
= \frac{ 1 + nx^{n - 1} (x - 1 - x/n) }{ (1 - x)^2 }
\end{align}

We thus have
\begin{align*}
\Esp{ \xi_n(y)^k } & = \frac{n + k + 1 }{ (k + 1)!}  \int_{ \crochet{0, 1}^k }  \frac{ \det\prth{ g_{n + k}( e^{2i\pi (\theta_m - \theta_j) } y^{-2} ) }_{1 \leq m, j \leq k } }{ \abs{ \Delta(e^{2i\pi \theta_1  } y\inv , \dots, e^{2i\pi \theta_k  } y\inv) }^2 }    d\thetab \\
                   & = \frac{n + k + 1 }{ (k + 1)!} y^{  k(k-1)  } \int_{ \crochet{0, 1}^k }  \frac{ \det\prth{ g_{n + k}( e^{2i\pi (\theta_m - \theta_j)  } y^{-2} ) }_{1 \leq m, j \leq k } }{ \abs{ \Delta(e^{2i\pi \theta_1  }, \dots, e^{2i\pi \theta_k }) }^2 }    d\thetab 
\end{align*}
hence
\begin{align*}
F_n(y, A) = A^{n + 1} \sum_{k \geq 0} \frac{ (-1)^k }{ k! \, (k + 1)!} (A y^{n + k - 1} )^k   \int_{ \crochet{0, 1}^k }  \frac{ \det\prth{ g_{n + k}( e^{2i\pi (\theta_m - \theta_j)  } y^{-2} ) }_{1 \leq m, j \leq k } }{ \abs{ \Delta(e^{2i\pi \theta_1  }, \dots, e^{2i\pi \theta_k }) }^2 }    d\thetab 
\end{align*}

\noindent \textbf{\underline{Main contribution :}} Using \eqref{Def:FunctionGn}, i.e. $ g_n(x) =  \frac{ 1 + nx^n - n x^{n - 1} - x^n }{ (1 - x)^2 } = \frac{ 1 - x^n }{(1 - x)^2} - \frac{ n x^{n - 1} }{ 1 - x } $, we have, when $ N = n + k \to \infty $, and $ z \in \Uu $
\begin{align*}
g_N\prth{ y^{-2} z } & = \frac{ 1 - t^N }{(1 - t)^2}\Big\vert_{t = y^{-2} z  } - \frac{ N t^{N - 1} }{ 1 - t } \Big\vert_{t = y^{-2} z  } 
                 = -\frac{N z^{N - 1} y^{-2(N - 1)} }{ 1 - y^{-2} z} + O\prth{ \frac{  y^{-2N}  }{  (1 - y^{-2} z)^2 } } \\
               & = -\frac{N z^{N - 1} y^{-2(N - 1)} }{ 1 - y^{-2} z} \prth{ 1 + O\prth{ \frac{1}{N  (1 - y^{-2} z) } } }
\end{align*}

Note that the implied constant in the $ O $ is independent of $z$ and $y$. We then get
\begin{align*}
\det\prth{ g_N (y^{-2} u_i \overline{u_j}) }_{1 \leq i, j \leq k} & =  \sum_{ \sigma \in \Sg_k  } \varepsilon(\sigma) \prod_{i = 1}^k g_N (y^{-2} u_i \overline{u_{ \sigma(j) } }) \\
              & = \sum_{ \sigma \in \Sg_k  } \varepsilon(\sigma) N^k y^{-2 (N - 1) k } \prth{ \prod_{i = 1}^k \frac{- 1}{ 1 - y^{-2} u_i \overline{u_{ \sigma(j) } } } } \times \\
              & \hspace{+4cm } \prth{ 1 + O\prth{ \frac{1}{N } \sum_{i = 1}^k \frac{1}{ 1 - y^{-2} u_i \overline{u_{\sigma(i) } } } } } \\
              & = (-N)^k  y^{-2 (N - 1) k } \det\prth{ \frac{1}{1 - y^{-2} u_i \overline{u_j} } }_{1 \leq i, j \leq k} \times \\
              & \hspace{+4cm }  \prth{ 1 + O\prth{ \frac{1}{N } \sum_{i, j = 1}^k \frac{1}{ 1 - y^{-2} u_i \overline{u_j } } } }
\end{align*}

Using the Cauchy determinant \eqref{Def:CauchyDeterminant} and the relation \eqref{Eq:VandermondeDuality}, we get for $ u_j \in \Uu $
\begin{align*}
\det\prth{ \frac{1}{1 - y^{-2} u_i \overline{u_j} } }_{1 \leq i, j \leq k} & = \det\prth{ \frac{1}{( y^{-2} \overline{u_j}) (y^2 u_j - u_i ) } }_{1 \leq i, j \leq k} \\
                  & =  y^{2k} \prod_{j = 1}^k u_j \, \frac{ \Delta(y^2 \ub)  \Delta(-\ub)  }{ \prod_{1 \leq i, j \leq k} (y^2 u_i - u_j ) }      \\
                  & =  (-1)^{ k(k + 1)/2 } y^{2k + k(k + 1)} \prod_{j = 1}^k u_j^k \, \frac{ \Delta(\ub) \overline{\Delta(\ub)}  }{ \prod_{1 \leq i, j \leq k} (y^2 u_i u_j\inv -1) u_j  }      \\
                  & =  (-1)^{ k(k + 1)/2 } y^{ k(k + 3)} \frac{\prod_{j = 1}^k u_j^k}{ \prod_{1 \leq i, j \leq k} (-u_j) } \, \frac{ \Delta(\ub) \overline{\Delta(\ub)}  }{ \prod_{1 \leq i, j \leq k} (1 - y^2 u_i u_j\inv  ) }      \\ 
                  & = (-1)^{ k(3k + 1)/2 }\,  y^{ k(k + 3)}   \frac{ \Delta(\ub) \overline{\Delta(\ub)} }{ \prod_{1 \leq i, j \leq k} (1 -  y^2 u_i \overline{u_j} ) }  
\end{align*}

It is moreover clear that the function $ \ub \mapsto \prod_{1 \leq i, j \leq k} (1 -  y^2 u_i \overline{u_j} )\inv  $ is integrable on $ \Uu^k $ as $ y < 1 $. We thus get 
\begin{align*}
F_n(y, A) & = A^{n + 1} \sum_{k \geq 0} \frac{(-1)^k }{ k! \,(k + 1)! }  (A y^{ n + k  - 1} )^k  \times \\
             & \hspace{+3cm} (-1)^{ \frac{k(3k + 1)}{2} } ( -(n + k) y^{ -2(n - 1) - 2k + k + 3} )^k \Bigg[ \int_{ \Uu^k }   \prod_{1 \leq i, j \leq k} \frac{   1  }{  1 -  y^2 u_i \overline{u_j}  }  \prod_{\ell = 1}^k   \frac{d^*u_\ell}{u_\ell } \\ 
             &  \hspace{+4cm} + O\prth{ \frac{ y^2}{n } \int_{ \Uu^k }   \prod_{1 \leq i, j \leq k} \frac{   1  }{  1 -  y^2 u_i \overline{u_j} }  \sum_{i, j} \frac{1}{1 - y^2 u_i u_j\inv} }   \prod_{\ell = 1}^k   \frac{d^*u_\ell}{u_\ell }  \Bigg] \\
             & = A^{n + 1} \crochet{ \sum_{k \geq 0} \frac{(-1)^{ \frac{k(7k + 1)}{ 2 } }  }{ k! \,(k + 1)! }  (A y^{ -n  } (n + k) )^k     \int_{ \Uu^k }   \prod_{1 \leq i, j \leq k} \frac{   1  }{  1 -  y^2 u_i \overline{u_j}  }  \prod_{\ell = 1}^k   \frac{d^*u_\ell}{u_\ell }  + O\prth{ \frac{ f(y) }{n  } } }
\end{align*}

One can check that a possible choice for $ f(y) $ is given by $ e^{y/(1 - y^2)} $. For the choice
\begin{align*}
A := \frac{y^n }{n}
\end{align*}
and using $ (-1)^{ \frac{k(7k + 1)}{2} } = (-1)^{ \frac{ k( k + 1)}{2} }  $, we obtain
\begin{align*}
F_n\prth{ y , \frac{ y^n  }{ n } } & =  \prth{ \frac{y^n }{n} }^{n + 1 } \crochet{ \sum_{k \geq 0} \frac{(-1)^{ \frac{ k (k + 1)}{2} }  }{ k! \,(k + 1)! } \prth{ \frac{n + k}{  n} }^k  \int_{ \Uu^k }    \prod_{ i, j = 1}^k\frac{ 1 }{ 1 -  y^2 u_i \overline{u_j} }  \prod_{\ell = 1}^k   \frac{d^*u_\ell}{u_\ell }  + O\prth{ \frac{ f(y) }{n^2 } } } \\
              & =   \frac{ y^{n(n + 1) } }{ n^{n + 1 } }  \crochet{  \sum_{k \geq 0} \frac{ (-1)^{ \frac{ k( k + 1)}{2} }     }{ k! \,(k + 1)! }    \int_{ \Uu^k } \prod_{ i, j = 1}^k \frac{   1  }{  1 -  y^2 u_i \overline{u_j} }  \prod_{\ell = 1}^k  \frac{d^*u_\ell}{u_\ell }  + O\prth{ \frac{F(y)}{n  } } }
\end{align*}
where $ F(y) $ is a function that can be made explicit.

$ $

\noindent \textbf{\underline{Remainder :}} Define 
\begin{align*}
\varepsilon_n(y ) := \frac{ n^{n + 1 } }{ y^{n(n + 1) } }  \int_{ y^n / n }^{+\infty } \Espr{n}{ e^{ -t y^n \!\! \int_0^1 \abs{ Z_n\prth{  e^{ 2i\pi \theta } y\inv } }^2 d\theta } } t^n dt 
\end{align*}
so that 
\begin{align*}
\Prob{ \rho_n \leq y } = \frac{ y^{n(n + 1) } }{ n^{n + 1 } }\crochet{ \Fe(y) + O\prth{ \frac{F(y)}{n} } + \varepsilon_n(y) }
\end{align*}

Let $ \gammab_{n + 1} $ be a random variable defined by $ \Prob{ \gammab_{n + 1} \geq x } := \int_x^{+\infty} e^{-t } t^n \frac{dt}{n! } $ (i.e. Gamma-distributed). Let us suppose that $ \gammab_{n + 1} $ is independent of $ (X_k)_{1 \leq k \leq n} $, the determinantal point process of kernel $ (z_1, z_2) \mapsto g_n(z_1 \overline{z_2}) $ on the unit disk (namely the eigenvalues of the truncated $ CUE(n + 1) $ random matrix whose \eqref{Def:TruncatedCUE} is the law). Then, we have
\begin{align*}
\varepsilon_n(y ) & =  \int_1^{+\infty } \Espr{n}{ e^{ - \frac{s}{n }  y^{2n} \!\! \int_0^1 \abs{ Z_n\prth{  e^{ 2i\pi \theta } y\inv } }^2 d\theta } } s^n ds \\
                & = \Esp{ \frac{ n! }{ \prth{  \frac{y^{2n}  }{n }  \!  \int_0^1 \abs{ Z_n\prth{  e^{ 2i\pi \theta } y\inv } }^2 d\theta }^{n + 1}   } \Unens{ \gammab_{n  + 1} \geq \frac{y^{2n}  }{n }  \!  \int_0^1 \abs{ Z_n\prth{  e^{ 2i\pi \theta } y\inv } }^2 d\theta }  } \\
                & \leq \Esp{ \frac{ n! n^{n + 1} }{ \prth{  y^{2n} \!  \int_0^1 \abs{ Z_n\prth{  e^{ 2i\pi \theta } y\inv } }^2 d\theta }^{n + 1}   }   }
\end{align*}

We have moreover
\begin{align*}
y^{2n} \!\! \int_0^1 \abs{ Z_n\prth{  e^{ 2i\pi \theta } y\inv } }^2 d\theta  = \int_0^1 \prod_{k = 1}^n  \abs{ 1 - e^{2 i \pi \theta } y X_k  }^2 d\theta = \sum_{k = 0}^n y^{2 k } \abs{ e_k(X_1, \dots, X_n) }^2  
\end{align*}
and, using \eqref{Thm:LimitLaw} for $ t = y\inv $, with $ G(t) := \prod_{k \geq 1} (1 - t^{2k} ) $, one gets for all $ t \in (0, 1) $
\begin{align*}
\Esp{ \frac{ 1 }{ \prth{ \sum_{ k = 0 }^n t^{4n - 2k } \abs{ e_k(X_1, \dots, X_n) }^2 }^{n + 1} } }  \equivalent{n}{+\infty} t^{n (n + 1) } \frac{ G(t) }{n!}
\end{align*}

In particular, using the fact that for all $ t \in (0, 1) $ and for all $ k \in \intcrochet{0, n} $, $ t^{2k - 2n} \geq t^{ 4n - 2k } $, one gets
\begin{align*}
\varepsilon_n(y ) & \leq \Esp{ \frac{ n! n^{n + 1} }{ \prth{  y^{2n} \!  \int_0^1 \abs{ Z_n\prth{  e^{ 2i\pi \theta } y\inv } }^2 d\theta }^{n + 1}   }  } =  \Esp{ \frac{ n! n^{n + 1} }{ \prth{  \sum_{k = 0}^n y^{ 2k - 2n } \abs{ e_k(X_1, \dots, X_n) }^2 }^{n + 1}   }   } \\
                & \leq  \Esp{ \frac{ n! n^{n + 1} }{ \prth{  \sum_{k = 0}^n y^{ 4n - 2k } \abs{ e_k(X_1, \dots, X_n) }^2 }^{n + 1} }  }  \equivalent{n}{+\infty} y^{n (n + 1) } n^{n +1} G(y)  \tendvers{n}{ + \infty } 0
\end{align*}
hence the result.
\end{proof}

\begin{remark}
It is clear that $ \Fe(y) \in \Rr $ since it writes as
\begin{align*}
\Fe(y) := \sum_{k \geq 0} \frac{ (-1)^{ \frac{ k( k + 1)}{2} } (1 - y^2)^{-k}   }{ k! \,(k + 1)! }    \int_{ \crochet{0, 1}^k } \prod_{ 1 \leq m < j \leq k} \frac{   1  }{  \abs{ 1 -  y^2 e^{2 i \pi (\theta_m - \theta_j) } }^2 }  \prod_{\ell = 1}^k  d \theta_\ell  
\end{align*}
\end{remark}


\begin{remark}
The function $ \Fe $ is similar to a Fredholm expansion since it involves a series with integrals, but one can ask about a more ``classical'' expression, for instance, a series expansion of the form $ \Fe(y) = \sum_{k \geq 0} a_k y^k $. Such an expression is possible if one writes the expansion of the Cauchy product in terms of Schur functions or power functions (see \cite[ch. 1.3]{MacDo}), i.e.
\begin{align*}
\prod_{i, j = 1}^k \frac{1}{1 - y^2 u_i \overline{u_j} } = \sum_{ n \geq 0 } y^{2n } h_n\crochet{\ub \overline{\ub}} 
\end{align*}
with (the notations are the ones in \cite[ch. 2]{HaimanMacDo})
\begin{align*}
h_n\crochet{\ub \overline{\ub}} := \sum_{\lambda \vdash n} \abs{ s_\lambda(\ub) }^2 = \sum_{\lambda \vdash n} \frac{\abs{ p_\lambda(\ub) }^2}{z_\lambda }
\end{align*}

Integrating these expressions and applying the Fubini theorem for double sums (wich is convergent due to the presence of the term $ \frac{1}{k! (k + 1)!} $ and the fact that the number of terms in $ h_n\crochet{\ub\overline{\ub}} $ is equal to $p(n)$, the number of partitions of an integer $n$, and the fact that $ p(n) = O(e^{C \sqrt{n} } ) $ for $ C = \pi \sqrt{ 2/3 } $ in virtue of a celebrated result of Hardy and Ramanujan), one gets the result. Nevertheless, the integrals $ \int_{\Uu^k } \abs{s_\lambda(\ub) }^2 \frac{d^*\ub}{\ub} $ or $ \int_{\Uu^k } \abs{p_\lambda(\ub) }^2 \frac{d^*\ub}{\ub} $ do not have a simple expression other than a combinatorial sum involving semi-standard tableaux or equivalent quantities. We leave to the interested reader the exercise to express such quantities.
\end{remark}


\section{Conclusion and perspectives}

In view of the previous results, a natural continuation of the problem is the following one : compute the (precise) \textit{transition deviations} around $ y = 1 $, in particular, one can conjecture that there exists a function $ \psi $ such that for all $ x > 0 $
\begin{align*}
\frac{1}{n} \log \Prob{ \rho_n \leq e^{- x/n } } \tendvers{n}{+\infty } - \psi(x)
\end{align*}
and more generally that there exists a function $ \Ge $ such that
\begin{align*}
\Prob{ \rho_n \leq e^{- x/n } } \equivalent{n}{+\infty } e^{ - n \psi(x) + o(n) } \Ge(x)
\end{align*}

Note that one can replace the term $ e^{- x/n } $ by $ 1 - \frac{x}{n} $. The problem of the precise transition deviations is then reminiscent of another problem of random polynomials, the computation of the \textit{persistence exponent}~; in the case of the real Kac polynomial, when $ \max_k \abs{Z_k} $ is replaced by the maximum absolute value of the real roots, see e.g. \cite{AldousFyodorov, BrayMajumdarSchehr, DemboPoonenShaoZeitouni, MajumdarSchehr} and references cited. The proof of theorem \ref{Theorem:MainTheorem} can be adapted to this setting, but the rescaling of the integrals involve the moments of the characteristic polynomial of a random truncated $ CUE(n) $ matrix in the \textit{microscopic scaling} (see \cite{ChhaibiNajnudelNikeghbali, KillipRyckman} for the case of the $CUE$) and requires additional care to extract its asymptotic behaviour. We plan to address this question in a subsequent publication.

\section{Annex : Expectation of products of characteristic polynomials}\label{Section:Annexe1}

We prove here \eqref{Eq:MomentsCharPol}. We want to compute 
\begin{align*}
\Espr{n}{ \prod_{\ell = 1 }^k \abs{ Z_n(u_\ell) }^2 } = \int_{ \Dd^n } \prod_{\ell = 1}^k \prod_{j = 1}^n \abs{ z_j - u_\ell  }^2  \abs{ \Delta(\zb) }^2 \frac{d\zb }{\pi^n  }
\end{align*}

One has 
\begin{align*}
\prod_{\ell = 1}^k \prod_{j = 1}^n \prth{ z_j - u_\ell }  = \frac{ \Delta(z_1, \dots, z_n, u_1, \dots, u_k) }{ \Delta(z_1, \dots, z_n) \Delta(u_1, \dots, u_k)} =: \frac{\Delta(\zb, \ub) }{\Delta(\zb) \Delta(\ub) }
\end{align*}
hence
\begin{align*}
\Espr{n}{ \prod_{\ell = 1 }^k \abs{ Z_n(u_\ell) }^2 } = \frac{1}{ \abs{ \Delta(\ub) }^2   } \int_{ \Dd^n }  \abs{ \Delta(\zb, \ub) }^2 \frac{d\zb }{\pi^n  }
\end{align*}

We now compute
\begin{align}\label{JointIntensityOnDisk}
\int_{ \Dd^n }   \abs{ \Delta(\zb, u_1, \dots, u_k) }^2   \frac{d\zb}{\pi^n}
\end{align}

This is, up to a multiplicative constant, the \textit{$ n $-joint intensity} of a random variable of size $ n + k $ with a joint law given by $ \Pp_n $ defined in \eqref{Def:TruncatedCUE} (see e.g. \cite{TaoGUE6} or \cite[ch. 3]{BenHougKrishnapurPeresVirag}). 

A classical method expresses $ \abs{ \Delta(\zb, u_1, \dots, u_k) }^2 $ as $ C \det\prth{ K(x_i, x_j) }_{i, j \leq n + k} $ for a certain ``kernel'' $ K $ autoreproduced for the underlying scalar product, here $ L^2( \mu) $, where
\begin{align*}
\mu(dz) = \Unens{z \in \Dd } \frac{dz}{\pi}
\end{align*}

Start by writing the square modulus of the Vandermonde determinant in the variable $ \xb := (x_1, \dots, x_m) $ as
\begin{align*}
\abs{ \Delta( \xb ) }^2 & =   \Delta( \xb ) \overline{ \Delta( \xb ) }  = \det\prth{ x_i^{j-1} }_{1 \leq i, j \leq m} \det\prth{ \overline{x_i^{j-1}} }_{1 \leq i, j \leq m} \\
				   & = \det\prth{ Q_{j - 1}(x_i) }_{1 \leq i, j \leq m} \, \det\prth{ \overline{Q_{j - 1}(x_i)} }_{1 \leq i, j \leq m} \\
				   & =  \prod_{j = 0}^{m-1} \norm{ Q_j }^2_{ L^2(\mu) } \det\prth{ \sum_{\ell = 1}^m \frac{Q_{\ell - 1}(x_i) \overline{ Q_{\ell - 1}(x_j) }}{ \norm{ Q_\ell }^2_{ L^2(\mu) } } }_{1 \leq i, j \leq m} \\
				   & =: \prod_{j = 0}^{m-1} \norm{ Q_j }^2_{ L^2(\mu) } \det\prth{ R_m(x_i, x_j) }_{1 \leq i, j \leq m}
\end{align*}
where the $ Q_i $ are monic (i.e. with the higher degree coefficient equal to 1) and satisfy $ \deg(Q_j) = j $. The fact that $ \det( x_i^{j-1} )_{1 \leq i, j \leq m} = \det\prth{ Q_{j - 1}(x_i) }_{1 \leq i, j \leq m} $ comes from the fact that one can perform linear combinations of lines or columns without changing its value.

The function $ R_m : (x, y) \mapsto \sum_{\ell = 0}^{m - 1} Q_\ell(x ) \overline{ Q_\ell(y) } / \norm{ Q_\ell }^2_{ L^2(\mu) } $ is said to be a kernel. As such a kernel satisfies the following properties 
\begin{enumerate}
\item Autoreproduction : $ \int_\Cc R_m(x,y) R_m(y, z) d\mu(y) = R_m(x, z) $, 
\item Trace : $ \hspace{+2.2cm} \int_\Cc R_m(x,x) d\mu(x) = m $  
\end{enumerate}
the following formula holds for all $ k \leq m $ (see \cite[Lemma 1]{TaoGUE6} or \cite[ex. 4.1.1]{BenHougKrishnapurPeresVirag}) 
\begin{align*}
\int_{\Cc } \det\prth{ R_m(x_i, x_j) }_{1 \leq i, j \leq k} d\mu(x_k)  =  (m - k + 1)  \det\prth{ R_m(x_i, x_j) }_{1 \leq i, j \leq k-1}
\end{align*}

Iterating, and defining $ N^{\uparrow k } := N(N + 1) \cdots(N + k - 1) $, we get for all $ \ell \leq k $
\begin{align}\label{AutoreprTraceFormula}
\int_{\Cc^\ell } \det\prth{ R_m(x_i, x_j) }_{1 \leq i, j \leq k} d\mu(x_k) 
 \cdots d\mu(x_{k - \ell + 1}) =  (m - k + 1)^{\uparrow \ell }  \det\prth{ R_m(x_i, x_j) }_{1 \leq i, j \leq k - \ell}
\end{align}

In order to compute \eqref{JointIntensityOnDisk} with this last formula, we need to find the kernel $ R_n $ associated to $ \mu $. One thus needs to norm the columns of the Vandermonde matrix $ (Q_i(x_j))_{i, j \leq n} $ since these polynomials are monic. For this, we compute
\begin{align*}
\int_\Cc z^\ell \overline{z^k} \, d\mu (z) = 2   \int_0^1 \int_0^1 r^{k + \ell} e^{2 i \pi (\ell - k)\theta } \,  rdr d\theta = \frac{ 2 }{k + \ell + 2} \Unens{k = \ell} = \frac{ 1 }{  k + 1 } \Unens{k = \ell}
\end{align*}

In particular, for $ k = \ell $,
\begin{align*}
\int_\Cc \abs{ z^k }^2 d\mu (z) = \frac{ 1 }{ k + 1 } 
\end{align*}

The normed polynomials are thus given for $ k \in \intcrochet{0, n - 1} $ by
\begin{align}\label{Def:NormedPolynomials}
\frac{ Q_k (z) }{ \norm{Q_k }_{L^2(\mu) } } = \sqrt{ k + 1 } \, z^k  
\end{align}
and the kernel is
\begin{align*}
R_n (u, v) := \sum_{k = 0}^{n-1} \frac{ Q_k(u) \overline{Q_k(v)}}{ \norm{Q_k }^2_{L^2(\mu) }  } = \sum_{k = 0}^{n - 1} (k + 1) ( u \bar{v} )^k   =:  g_n\prth{  u \bar{v} }
\end{align*}
where $ g_n $ is defined in \eqref{Def:FunctionGn}.


Finally, one can write 
\begin{align*}
\Delta(\xb) = \det\prth{ x_i^{j-1} }_{1 \leq i, j \leq m} & = \prth{ \prod_{j = 1}^m \frac{ 1 }{ \sqrt{ j } } } \det\prth{ x_i^{j-1} \sqrt{ j} }_{1 \leq i, j \leq m}  = \frac{ 1 }{ \sqrt{  m! } } \det\prth{ \frac{Q_{j-1} (x_i)}{ \norm{Q_{j - 1} }_{L^2(\mu)} } }_{1 \leq i, j \leq m}
\end{align*}
which implies that
\begin{align}\label{Eq:Vandermonde=Noyau}
\abs{ \Delta( \xb ) }^2 & =   \Delta( \xb ) \overline{ \Delta( \xb ) }  = \frac{ 1 }{  m! } \det\prth{ R_m (x_i, x_j) }_{1 \leq i, j \leq m}
\end{align}

Setting $ \xb := (\zb, u_1, \dots, u_k) $, we can write \eqref{JointIntensityOnDisk} as
\begin{align*}
\int_{ \Dd^n }    \abs{ \Delta(\zb, u_1, \dots, u_k) }^2  \, \frac{d\zb}{ \pi^n  } & = \int_{ \Cc^n }  \abs{ \Delta( \zb , \ub  ) }^2  \,  d \mu^{\otimes n} (\zb) \\
	                    & =  \frac{ 1 }{  (n + k)! }   \int_{ \Cc^n }  \det\prth{ R_{n + k} (x_i, x_j) }_{1 \leq i, j \leq n  + k} \,  d \mu^{\otimes n}   (\zb ) \\
	                    & =  \frac{  n! }{ (n + k)!  }  \,    \det\prth{ R_{n + k} (u_i, u_j) }_{1 \leq i, j \leq k} \mbox{ by \eqref{AutoreprTraceFormula} } \\
	                    & =    \frac{ 1  }{ (n + k) (n + k - 1) \dots (n + 1) }   \det\prth{ g_{n + k} (u_i \overline{u_j} ) }_{1 \leq i, j \leq k}
\end{align*}
which is \eqref{Eq:MomentsCharPol}.

\section*{Acknowledgements}

The author thanks Gernot Akemann, Olivier H\'enard, Mario Kieburg, Igor Krasovski, Ashkan Nikeghbali, Dan Romik, Nick Simm, Oleg Zaboronski and Ofer Zeitouni for interesting discussions, questions and remarks concerning previous versions of this work.

A particular thanks is given to Michael Cranston and Elliot Paquette for technical discussions and computational help, and to Rapha\"el Butez for providing the author with an earlier version of \cite{ButezPAG} and interesting discussions that followed. The author acknowledges moreover the organisers of the conference ``Random matrix theory and strongly correlated systems'' held at the university of Warwick the 21-24 March 2016 that allowed him to meet with several of the aforementionned persons.

During the redaction of this work, the author was supported by the Schweizerischer Nationalfonds PDFMP2 134897/1 and by the EPSRC grant EP/L012154/1.


\bibliographystyle{amsplain}


\end{document}